\documentclass[11pt,a4paper,onehalfspacing]{article}
\usepackage{amssymb}
\usepackage{rotating}
\usepackage{amsfonts}
\usepackage{amsmath}

\usepackage{amsthm}
\usepackage{caption}
\usepackage{float}
\usepackage{graphicx}
\usepackage[colorlinks=true, urlcolor=bleu, pdfborder={0 0 0}]{hyperref}
\usepackage{enumerate}
\usepackage{multicol}
\usepackage{array}
\usepackage{bbm}
\usepackage{mathtools}
\usepackage{leftidx}
\usepackage[para]{footmisc}

\newtheorem{theorem}{Theorem}
\newtheorem{proposition}[theorem]{Proposition}
\newtheorem{lemma}[theorem]{Lemma}
\newtheorem*{lemma*}{Lemma}
\newtheorem*{assumption*}{Assumption}
\newcounter{assumptionc}

\newtheorem{corollary}[theorem]{Corollary}
\newtheorem{definition}[theorem]{Definition}

\newtheorem{remark}[theorem]{Remark}

\numberwithin{equation}{section}
\numberwithin{theorem}{section}
\newtheorem{ex}[theorem]{Example}

\newtheorem{subassumption}{Assumption}[assumptionc]

\newenvironment{assumption+}
 {\ifnum\value{subassumption}=0 \stepcounter{assumptionc}\fi\subassumption}
 {\endsubassumption}

\def\text#1{\hbox{#1}}

\def\build #1_#2{\mathrel{\mathop{\kern 0pt #1}\limits_\zs{#2}}}
\newcommand{\zs}[1]{{\mathchoice{#1}{#1}{\lower.25ex\hbox{$\scriptstyle#1$}}
{\lower0.25ex\hbox{$\scriptscriptstyle#1$}}}}

\numberwithin{equation}{section}

\usepackage[top=1in, bottom=1in, left=1.in, right=0.9in]{geometry}

\newtheorem*{example*}{Example}

\newcommand{\tr}{\mathrm{tr}}
\newcommand{\PP}{\mathbb{P}}
\newcommand{\EX}{\mathbb{E}}
\newcommand{\Real}{\mathbb{R}}
\newcommand{\N}{\mathbb{N}}

\newcommand{\BB}{\mathcal{B}}

\newcommand{\diff}{\mathrm{d}}
\newcommand\norm[1]{\left\lVert#1\right\rVert}

\renewcommand{\norm}[1]{\left\lVert#1\right\rVert}

\makeatletter
\newenvironment{myproof}[1][\proofname]{%
  \par\pushQED{\qed}\normalfont%
  \topsep6\p@\@plus6\p@\relax
  \trivlist\item[\hskip\labelsep\bfseries#1\@addpunct{.}]%
  \ignorespaces
}{%
  \popQED\endtrivlist\@endpefalse
}
\newcommand{\myitem}[1]{%
	\item[#1]\protected@edef\@currentlabel{#1}%
}
\makeatother

\def\dfrac{\displaystyle\frac}

\def\FF{\mathcal{F}}

\usepackage{accents}

\newcommand{\LL}{\mathcal{L}}

\allowdisplaybreaks

\begin{document}

\title{The $C^{0,1}$ It{\^o}-Ventzell formula for weak Dirichlet processes}
\author{Felix Fie{\ss}inger\footnote{University of Ulm, Institute of Insurance Science and Institute of 		Mathematical Finance, Ulm, Germany. Email: felix.fiessinger@uni-ulm.de}  \, and Mitja Stadje\footnote{University of Ulm, Institute of Insurance Science and Institute of Mathematical Finance, Ulm, Germany. Email: mitja.stadje@uni-ulm.de}}
\date{\today}
\maketitle
\begin{abstract}
	This paper proves an extension of the It{\^o}-Ventzell formula that applies to stochastic flows in $C^{0,1}$ for continuous weak Dirichlet processes. We apply this theorem, for example, to give a representation result for strong solutions of time-dependent elliptic SPDEs, to derive formulas for quadratic variations, and to relax assumptions in a financial mathematics context.
\end{abstract}




\section{Introduction}

In this paper, we show an extension of the It{\^o}-Ventzell formula for continuous weak Dirichlet processes, such that the underlying stochastic flow only needs to be in $C^1$ instead of being in $C^2$ in the spacial argument under an additional mild assumption. The It{\^o}-Ventzell formula is a generalization of the It{\^o} formula for stochastic flows. Specifically, let $X=M+A$ be a continuous semimartingale, where $M$ is a local martingale and $A$ is a process of finite variation, and let $F(t,x)$ be a stochastic flow with local characteristics $(\beta(t,x),\gamma(t,x))$, i.e., $F(t,x)$ is a random field and admits the It{\^o} dynamics
\begin{align*} 
	F(t,x) = F(0,x) + \int_0^t \beta (r,x) \diff r + \int_0^t \gamma (r,x) \diff W_r,
\end{align*}
where $W$ is a Brownian Motion. Moreover, we assume that $F$ is a continuous progressively measurable process for fixed $x$ and twice differentiable in $x$ for fixed $t$ and $\omega$. Note that the dependence of $F$, $\beta$, and $\gamma$ on $\omega$ is suppressed. Then the It{\^o}-Ventzell formula states that $F(t,X_t)$ is also a continuous semimartingale and has the following decomposition:
\begin{align*}
F(t,X_t) =& \ F(0,X_0) + \int_0^t \gamma (s,X_s) \diff W_s + \int_0^t F_x (s,X_s) \diff M_s + \BB^X (F)_t,
\end{align*}
where
\begin{align*}
	\BB^X (F)_t = \int_0^t \beta (s,X_s) \diff s + \int_0^t F_x (s,X_s) \diff A_s + \dfrac{1}{2} \int_0^t F_{xx} (s,X_s) \diff [X]_s + \int_0^t \gamma_x (s,X_s) \diff [X,W]_s. 
\end{align*}
The It{\^o}-Ventzell formula links stochastic processes to stochastic partial differential equations (SPDEs) and has various applications. It is used, for instance, to analyze various optimal control problems, to determine localization errors of SPDEs, for generalized stochastic forward integrals, for forward performance processes, and to develop pricing models. We apply our new version of the It{\^o}-Ventzell formula to derive stochastic calculus formulas and to obtain a representation result for strong solutions of time-dependent stochastic elliptic PDEs. Here, we generalize naturally the solution theory of SPDEs to the setting of the paper. Moreover, we apply our result in a financial mathematics context to obtain the dynamics of the portfolio of a large investor whose trading influences prices.

Multiple works generalize the classical It{\^o} formula in different directions, see, for instance, F{\"o}llmer et al. \cite{follmer1995quadratic}, Elworthy et al. \cite{elworthy2007generalized} or da Prato et al. \cite{da2019mild}. We focus on the extension discovered by Ventzell \cite{ad1965equations} in 1965, the It{\^o}-Ventzell formula, which was generalized later by Kunita \cite{kunita1997stochastic}. Since then, multiple works extended the It{\^o}-Ventzell formula, see, for instance, Flandoli and Russo \cite{flandoli2002generalized}, Krylov \cite{krylov2009wentzell}, Karachanskaya \cite{karachanskaya2011generalization}, Kohlmann et al. \cite{Kohlmann}, or dos Reis and Platonov \cite{dos2022ito}.

Our results hold for weak Dirichlet processes, the sum of a local martingale and an adapted process with zero covariation with all continuous local martingales, including semimartingales as a special case. Since the classical theorem of stochastic integration is defined only with respect to semimartingales, we use an extension of stochastic integration developed by Russo and Vallois \cite{russo1991integrales,russo1993forward,russo1995generalized,russo1995generalized,russo1996ito,Russo}, Coquet et al. \cite{coquet2006natural}, and Bandini and Russo \cite{bandini2017weak,bandini2022weak}. 
Considering weak Dirichlet processes is not unusual in generalizations of It{\^o}'s-formula which require milder smoothness conditions on the function $F$, see Gozzi and Russo \cite{Gozzi} who generalized the classical It{\^o} formula to $C^{0,1}$, Bouchard et al. \cite{bouchard2022} for the functional It{\^o} formula, or Bouchard and Vallet \cite{bouchard2021dupire} for the It{\^o}-Dupire formula. 

The paper is structured as follows: In Section \ref{sec: basic definition}, we introduce weak Dirichlet processes, generalized stochastic integration, and stochastic flows. Section \ref{sec: main result} gives our result, namely the extension of the It{\^o}-Ventzell formula to stochastic flows in $C^{0,1}$, while the proof of this extension is shown in Section \ref{sec: proof of main result}. Section \ref{sec: applications} gives some applications of our extended version of the It{\^o}-Ventzell formula.

\section{Basic Definitions} \label{sec: basic definition}

Let $(\Omega,\FF,(\FF_t)_{t \in [0,T]},\PP)$ be a stochastic basis where the filtration $\FF$ is generated by a Brownian Motion $W$ satisfying the usual conditions and $[0,T]$ is the observation period.
We define $f_x(t,x)$ to be the partial derivative of $f$ in the second argument and denote higher derivatives similarly. 
In this paper, we use the space of progressively measurable and continuous processes equipped with the topology of uniform convergence on compacts in probability. We denote by $C^{n,m+\alpha}$ the set of functions which are $n$-times continuously differentiable in the first and $m$-times in the second argument where the $m$th derivative is $\diff \PP \times \diff t$-a.s. locally $\alpha$-H{\"o}lder continuous with $\alpha>0$. If the $m$th derivative is just continuous, we simply write $C^{n,m}$, and $C^{0,0}$ denotes the set of all continuous functions. Furthermore, we denote by $\LL^1 (\mu)$ (resp. $\LL^2 (\mu)$) the space of integrable (resp. square integrable) random variables/measurable functions with respect to a measure $\mu$.
\begin{definition}
	The sequence of progressively measurable processes $(H^n)$ converges uniformly on compacts in probability ($ucp$) to a process $H$ if $\sup_{s \in [0,t]} \left| H^n_s -H_s \right| \xrightarrow{n \to \infty} 0$ in probability for all $t>0$.
\end{definition}

Our processes may not be semimartingales and may not admit a quadratic (co-)variation. For this purpose, we need generalizations of these concepts, see Gozzi and Russo \cite[p.1567]{Gozzi}:

\begin{definition} \label{def: quadratic variation alternative definition}
	Let $X$ and $Y$ be two continuous processes of the same dimension. Then, we define for all $t \in [0,T]$:
	\begin{align*}
		[X,Y]_t &= \lim\limits_{\varepsilon \downarrow 0} \int_0^t \dfrac{(X_{r+\varepsilon}-X_r)(Y_{r+\varepsilon}-Y_r)}{\varepsilon} \diff r, \\
		\int_0^t X_r \diff^- Y_r &= \lim\limits_{\varepsilon \downarrow 0} \int_0^t X_r \dfrac{Y_{r+\varepsilon}-Y_{r}}{\varepsilon} \diff r,
	\end{align*}
	if those limits exist in the sense of $ucp$-convergence. We call $[X,Y]_t$ the covariation, and $\int_0^t X_r \diff^- Y_r$ the forward integral. We use the short notation for the covariation: $[X]_t := [X,X]_t$. 
\end{definition}

Definition \ref{def: quadratic variation alternative definition} ensures that the covariation and the forward integral are again continuous processes and that the covariation is bilinear and symmetric. In particular, the (forward) integral against the covariation is well-defined. For a more detailed explanation of the covariation and the integral against it, see Russo and Vallois \cite[pp. 83-97]{russo1995generalized}.

\begin{remark} \label{rem: constant continuation}
	Due to Definition \ref{def: quadratic variation alternative definition}, the functions inside of the covariation and the integral need to be defined on $[0,T+\varepsilon]$ instead of $[0,T]$. Consequently, we use the convention that these processes or functions are constantly continued on $(T,T+\varepsilon]$, e.g., we define $X_r = X_T$ for all $r>T$. 
\end{remark}

Definition \ref{def: quadratic variation alternative definition} ensures that the covariation and the forward integral are continuous. If $X$ and $Y$ are semimartingales, then the concepts above coincide with the quadratic covariation and the usual notion of a stochastic integral. While semimartingales are of finite quadratic variation, this is, of course, not true for general continuous processes. Hence, we define the following:

\begin{definition}
	Let $X$ be a continuous and adapted process. If $[X]_t < \infty$ for all $t \in [0,T]$, then we call $X$ a finite quadratic variation (FQV) process.
\end{definition}

The following definition is standard, see, for instance, Russo and Vallois \cite{russo1991integrales,russo1993forward,russo1995generalized,russo1995generalized,russo1996ito,Russo}, Coquet et al. \cite{coquet2006natural}, Bandini and Russo \cite{bandini2017weak,bandini2022weak} or Gozzi and Russo \cite[p.1570]{Gozzi}:
\begin{definition}
	A real continuous and adapted process $X$ is called a weak Dirichlet process if $X$ is decomposable into a local martingale $M$ and a weak zero energy process $A$, i.e., $X=M+A$. \\
	For the process $A$, it holds: $A_0=0$ and $[A,N]_t=0$ for all continuous local martingales $N$ and $t \in [0,T]$. 
\end{definition}
In general, this decomposition is not unique. However, when $M$ is required to be continuous, there exists a unique decomposition. 
In the following, we always use the latter decomposition. In particular, $A$ is continuous, since by assumption $X$ and $M$ are. If $X$ is additionally a semimartingale, $A$ is a continuous finite variation process.

\begin{ex}
	The following processes are weak-Dirichlet processes according to Gozzi and Russo \cite[pp.1570-1571]{Gozzi}, and Errami and Russo \cite[p.267]{errami2003n}:
	\begin{itemize}
		\item all semimartingales,
		\item all Dirichlet processes $X$, i.e., $X$ is decomposable into a local martingale and a zero quadratic variation process,
		\item the process $X_t = \int_0^t G(t,r) \diff M_r$, where $G: \{ 0 \leq r \leq t \leq 1\} \to \Real$ is a continuous $(\FF_r)$-measurable random field and $M$ is a local continuous $(\FF_t)$-martingale.
	\end{itemize}
\end{ex}

Contrary to the classical It{\^o} formula, the It{\^o}-Ventzell formula considers stochastic flows, which are defined in the following definition. Similar, but slightly different definitions can be found, e.g., in Kohlmann et al. \cite[p.7]{Kohlmann} or dos Reis and Platonov \cite[p.6]{dos2022ito}.
\begin{definition} \label{def: stochastic flow}
	Let $F(t,x)=F(\omega,t,x): \Omega \times [0,\infty) \times \Real^d \to \Real$ be a stochastic flow with local characteristics $(\beta(t,x),\gamma(t,x))$, i.e., $F$ is a random field which admits the It{\^o} dynamics
	\begin{align} \label{eq: definition formula F}
		F(t,x) = F(0,x) + \int_0^t \beta (r,x) \diff r + \int_0^t \gamma (r,x) \diff W_r,
	\end{align}
	where
	\begin{enumerate}[(a)]
		\item $W$ is the Brownian motion which generates the filtration $\FF$,
		\item $F(\cdot,x)$ is a progressively measurable process for all fixed $x \in \Real^d$,
		\item $\beta: \Omega \times \Real_{\geq 0} \times \Real^d \to \Real$ and $\gamma: \Omega \times \Real_{\geq 0} \times \Real^d \to \Real^d$ are previsible random fields,
		\item $\gamma(t,x)$ is $\diff \PP \times \diff t$-a.s. locally $\alpha$-H{\"o}lder continuous in $x$ for an $\alpha >0$, and
		\item \label{def, item: integrability}for any $K \subset \Real^d$ compact and for any $S \geq 0$ it holds a.s. that $\int_0^S  \sup_{x \in K} |\beta(t,x)| \diff t < \infty$, and $\int_0^S  \sup_{x \in K} |\gamma(t,x)|^2 \diff t < \infty$.
	\end{enumerate}
\end{definition}
The dependence on $\omega$ will be suppressed throughout this paper unless it is necessary for the steps of a proof.

\begin{remark}
	The local $\alpha$-H{\"o}lder continuity of $\gamma$ is necessary to ensure through Kolmogorov's continuity theorem that the stochastic integral $\int_0^t \gamma (r,x) \diff W_r$ is a.s. well-defined for all $(t,x)$. 
	We refer to Kunita \cite{kunita1997stochastic} for details.
\end{remark}

\section{Main Result} \label{sec: main result}

The following result states the It{\^o}-Ventzell formula for stochastic flows in $C^{0,1}$:

\begin{theorem} \label{th: main result}
	Let $X$ be a continuous weak Dirichlet FQV process and $F \in C^{0,1}$ be a stochastic flow with local characteristics $(\beta(t,x),\gamma(t,x))$. Furthermore, we assume: 
	\begin{enumerate}
		\myitem{1a} \label{th, item: beta integrability} $\int_0^T \sup_{s \in [0,T]} |\beta (s,X_t)| \diff t < \infty$ a.s., or
		\myitem{1b} \label{th, item: beta integral expected value} there exist a $B_1>0$ such that $\int_0^T \sup_{s \in [0,T]} \EX \left[ \left| \beta(s,X_t)\right| \right] \diff t \leq B_1$.
		\myitem{2} \label{th, item: gamma lemma} for all $K \subset \Real^d$ compact and $p \geq 1$, we suppose the existence of a $C_p \in \Real$ such that for any $x,y \in K$ holds that $\EX \left[\sup_{t \in [0,T]} |\gamma(t,x)-\gamma(t,y)|^p\right] \leq C_p |x-y|^p$ and $\EX \left[\sup_{t \in [0,T]} |\gamma(t,0)|^p \right] < \infty$.
		\myitem{3} \label{th, item: gamma continuity} $\gamma(t,X_t)$ is a.s. c{\`a}dl{\`a}g (right continuous with left limits).
	\end{enumerate}

	Then it holds for all $t \in [0,T]$:
	\begin{align*}
		F(t,X_t) = F(0,X_0) + \int_0^t \gamma (r,X_r) \diff W_r + \int_0^t F_x (r,X_r) \diff M_r + \BB^X (F)_t,
	\end{align*}
	where $\BB^X$ is linear and continuous (in $F$) with respect to the topology of $ucp$ convergence and $F(t,X_t)$ is again a continuous weak Dirichlet process. Moreover,
	\begin{enumerate}[(a)]
		\item \label{th, item: proof a}$\BB^X$ is a weak zero energy process, and
		\item \label{th, item: proof b}if $F(t,x)\in C^{0,2}$, $\beta(t,x) \in C^{0,1}$, and $\gamma(t,x) \in C^{0,1+\alpha}$ for an $\alpha \in (0,1)$, then it holds:
		\begin{align*}
			\BB^X (F)_t =& \int_0^t \beta (r,X_r) \diff r + \int_0^t F_x (r,X_r) \diff^{-} A_r + \int_0^t \gamma_x (r,X_r) \diff [X,W]_r + \frac{1}{2} \int_0^t F_{xx} (r,X_r) \diff [X]_r.
		\end{align*}
	\end{enumerate}
\end{theorem}

In the last two integrals of Theorem \ref{th: main result} \eqref{th, item: proof b} both the integrand and the integrator are matrices. These integrals are, as usual, defined as the trace of the matrix of the corresponding integrals. The assumptions of Theorem \ref{th: main result} are satisfied (through Assumption \ref{th, item: beta integrability}) 
if the mappings $(t,x) \mapsto \beta(t,x)$ and $(t,x) \mapsto \gamma(t,x)$ are a.s. continuous, $\sup_{t \in [0,T]} |\gamma(t,0)|$ is $p$-integrable for all $p \in [1,\infty)$ and $x \mapsto \gamma(\cdot,x)$ is a.s. locally Lipschitz continuous with a Lipschitz constant independent of $t$ which is integrable in $\omega$.

The theorem generalizes the result of Gozzi and Russo \cite[pp.1571-1574]{Gozzi} to stochastic flows. 

\begin{corollary}
	Let the assumptions of Theorem \ref{th: main result} hold. Moreover, assume that $F(\cdot,X)$ is a local martingale. Then, it holds for all $t \in [0,T]$:
	\begin{align*}
		F(t,X_t) = F(0,X_0) + \int_0^t \gamma (r,X_r) \diff W_r + \int_0^t F_x (r,X_r) \diff M_r.
	\end{align*}
\end{corollary}

\begin{proof}
	Using Theorem \ref{th: main result}, we get:
	\begin{align*}
		\BB^X (F)_t = F(t,X_t) - F(0,X_0) - \int_0^t \gamma (r,X_r) \diff W_r - \int_0^t F_x (r,X_r) \diff M_r.
	\end{align*}
	Since $F(\cdot,X)$, $M$, and $W$ are local martingales, $\BB^X (F)$ is a local martingale as well, and Theorem \ref{th: main result} states that $\BB^X (F)$ is a continuous weak zero energy process. Thus, $[\BB^X (F)]=0$ which implies that $\BB^X (F) \equiv 0$. Hence, Theorem \ref{th: main result} finalizes the proof.
\end{proof}

\begin{remark}
	Theorem 3.1 in Krylov \cite{krylov2009wentzell} provides a distribution-valued version of the It{\^o}-Ventzell formula. The main difference between his and our version of the It{\^o}-Ventzell formula is that in Theorem 3.1 in Krylov \cite{krylov2009wentzell}, one needs the existence of a second derivative in the space direction (in the distributional sense) of the function $F$, while we only need the first derivative. On the other hand, Krylov \cite{krylov2009wentzell} considers possibly infinitely many space directions and does not need joint continuity in time and space directions and the previsibility of the characteristics. Contrary, we need less regularity for the characteristics, i.e., we do not need to assume that $\beta$ and $\gamma$ are continuous. Regarding the integrability, our Theorem \ref{th: main result} and Theorem 3.1 in Krylov \cite{krylov2009wentzell} have somewhat similar assumptions.
\end{remark}

\section{Proof of Theorem \ref{th: main result}} \label{sec: proof of main result}

Before proving the theorem, we need some preliminary results. First of all, we assume $d=1$ to simplify the notation. In the whole section, we use the convention from Remark \ref{rem: constant continuation} for the constant continuation of processes and functions. 

By recalling Proposition 1.2 of Russo and Vallois \cite[p.89]{russo1995generalized} it holds:
\begin{lemma} \label{rem: cadlag ucp}
	Let $X$ and $Y$ be two FQV processes and $H$ be a c{\`a}dl{\`a}g process defined on the interval $[0,T]$. Then, it holds:
	\begin{align*}
		\int_0^t H(r) \frac{(X_{r+\varepsilon}-X_{r})(Y_{r+\varepsilon}-Y_{r})}{\varepsilon} \diff r \xrightarrow[ucp]{\varepsilon \to 0} \int_0^t H(r) \diff [X,Y]_r
	\end{align*}
\end{lemma}

\begin{lemma} \label{lem: ucp convergence double}
	Let $X$, $Y$ be two FQV processes, $t \in [0,T]$. Moreover, let $Z^{\varepsilon}$ be a measurable process in $r$ for all $\varepsilon>0$ such that  $\sup_{r \in [0,T]} \left|Z_r^{\varepsilon} \right| \xrightarrow[a.s.]{\varepsilon \to 0} 0$. Then:
	\begin{align*}
		\int_0^t Z^{\varepsilon}_r \frac{(X_{r+\varepsilon}-X_{r})(Y_{r+\varepsilon}-Y_{r})}{\varepsilon} \diff r \xrightarrow[ucp]{\varepsilon \to 0} 0.
	\end{align*}
\end{lemma}

\begin{proof}
	First, we note that $Z^{\varepsilon}$ is a.s. bounded for all $\varepsilon>0$ (with a bound possibly depending on $\omega$), i.e., the integral is well-defined. Then, we get: 
	\begin{align*}
		\sup_{t \in [0,T]} \left| \int_0^t Z^{\varepsilon}_r \frac{(X_{r+\varepsilon}-X_{r})(Y_{r+\varepsilon}-Y_{r})}{\varepsilon} \diff r \right| &\leq \int_0^T \left| Z^{\varepsilon}_r \frac{(X_{r+\varepsilon}-X_{r})(Y_{r+\varepsilon}-Y_{r})}{\varepsilon}  \right| \diff r \\
		&\leq \sup_{r \in [0,T]} \left| Z^{\varepsilon}_r \right| \cdot \int_0^T \left| \frac{(X_{r+\varepsilon}-X_{r})(Y_{r+\varepsilon}-Y_{r})}{\varepsilon} \right| \diff r,
	\end{align*}
	and the lemma follows from Lemma \ref{rem: cadlag ucp}.
\end{proof}

For the following lemma, we need the concept of boundedness in probability, see, e.g., van der Vaart \cite{van2000asymptotic}: 
\begin{definition}
	Let $I$ be an index set. A set of random variables $(X_\delta)_{\delta \in I}$ is bounded in probability (also known as uniformly tight) if there exists for all $\varepsilon>0$ an $M>0$ such that $\sup_{\delta \in I} \PP (|X_\delta|>M)<\varepsilon$.
\end{definition}

Note that $\LL^1$ boundedness implies via Markov's inequality boundedness in probability. 

\begin{lemma} \label{lem: ucp convergence single L1}
	Let $X$ be a continuous FQV process, $t \in [0,T]$, and $Z^{\varepsilon}$ be a measurable and a.s. bounded process for all $\varepsilon>0$ such that $\int_0^T |Z_r^{\varepsilon}| \diff r$ is bounded in probability in $\varepsilon>0$. Then:
	\begin{align*}
		\int_0^t Z^{\varepsilon}_r \left(X_{r+\varepsilon}-X_{r}\right) \diff r \xrightarrow[ucp]{\varepsilon \to 0} 0.
	\end{align*}
\end{lemma}

\begin{proof}
	Let $\omega \in \Omega$ be arbitrary such that $X(\omega)$ is continuous. 
	Thus, $X(\omega)$ is uniformly continuous on $[0,T]$, i.e., $\sup_{r \in [0,T]} \left| X_{r+\varepsilon} (\omega)-X_{r} (\omega) \right| \xrightarrow[a.s.]{\varepsilon \to 0} 0$. Hence, we get:
	\begin{align*}
		\sup_{t \in [0,T]} \left| \int_0^t Z^{\varepsilon}_r \left(X_{r+\varepsilon}-X_{r}\right) \diff r \right| &\leq \int_0^T \left| Z^{\varepsilon}_r \left(X_{r+\varepsilon}-X_{r}\right) \right| \diff r \\
		&\leq \int_0^T \left| Z^{\varepsilon}_r \right| \diff r \cdot \sup_{r \in [0,T]} \left| X_{r+\varepsilon} (\omega)-X_{r} (\omega) \right| \xrightarrow[\PP]{\varepsilon \to 0} 0.
	\end{align*}
\end{proof}

\begin{lemma} \label{lem: ucp convergence single}
	Let $X$ be a FQV process, $t \in [0,T]$, and $Z^{\varepsilon}$ be a measurable and a.s. bounded process for all $\varepsilon>0$ such that $\int_0^T \left|Z_r^{\varepsilon}\right|^2 \diff r \xrightarrow[\PP]{\varepsilon \to 0} 0$. Then:
	\begin{align*}
		\int_0^t Z^{\varepsilon}_r \frac{X_{r+\varepsilon}-X_{r}}{\sqrt{\varepsilon}} \diff r \xrightarrow[ucp]{\varepsilon \to 0} 0.
	\end{align*}
\end{lemma}

\begin{proof}
	We get with Cauchy-Schwarz:
	\begin{align*}
		\sup_{t \in [0,T]} \left| \int_0^t Z^{\varepsilon}_r \frac{X_{r+\varepsilon}-X_{r}}{\sqrt{\varepsilon}} \diff r \right| \leq \int_0^T \left| Z^{\varepsilon}_r \frac{X_{r+\varepsilon}-X_{r}}{\sqrt{\varepsilon}} \right| \diff r \leq \sqrt{\int_0^T \left| Z_r^{\varepsilon} \right|^2 \diff r} \sqrt{\int_0^T \frac{(X_{r+\varepsilon}-X_{r})^2}{\varepsilon} \diff r}.
	\end{align*}
	The first factor converges in probability to $0$ for $\varepsilon \to 0$ due to the assumption of the lemma. The second factor converges in probability to $\sqrt{[X]_T}$ for $\varepsilon \to 0$ by Definition \ref{def: quadratic variation alternative definition}.
\end{proof}

\begin{lemma} \label{lem: substitution}
	In the setting of Theorem \ref{th: main result}, it holds for all $t \in [r,T]$ and for all $\FF_r$-measurable random variables $G$ such that $\int_0^t |\beta(s,G)| \diff s < \infty$ a.s. and $\int_0^t |\gamma(s,G)|^2 \diff s < \infty$ a.s.:
	\begin{align*}
			\left. \int_r^t \beta(s,x) \diff s \right|_{x=G} &= \int_r^t \beta (s,G) \diff s,&
			\left. \int_r^t \gamma(s,x) \diff W_s \right|_{x=G} &= \int_r^t \gamma (s,G) \diff W_s.
		\end{align*}
\end{lemma}

\begin{proof}
	The first property follows directly from the $\omega$-wise definition of the ``$\diff s$''-integral. The second property is an immediate consequence of Lemma 3.3 in Russo and Vallois \cite{russo1996ito}.
\end{proof}

\begin{lemma} \label{lem: show old assumption}
	Consider the setting of Theorem \ref{th: main result} and assume that $X$ takes only values in a compact set $K \subset \Real$. Then Assumption \ref{th, item: gamma lemma} implies that there exist $B_2>0$ such that $\sup_{s \in [0,T]} \EX \left[\gamma(s,X_t)^{2+\varrho}\right] \leq B_2$ for all $t \in [0,T]$ and for all $\varrho>0$.
\end{lemma}

\begin{proof}
	The lemma follows from Assumption \ref{th, item: gamma lemma} in Theorem \ref{th: main result}.
\end{proof}


Now, we are ready to prove our main result:

\begin{myproof}[Proof of Theorem \ref{th: main result}]
	We can write $\BB^X$ as
	\begin{align*}
		\BB^X (F)_t = F(t,X_t) - F(0,X_0) - \int_0^t \gamma (r,X_r) \diff W_r - \int_0^t F_x (r,X_r) \diff M_r.
	\end{align*}
	Hence, $\BB^X$ is linear and continuous (in $F$) with respect to the $ucp$ topology. The property that $F(t,X_t)$ is a weak Dirichlet process follows directly from $\mathcal{B}^X$ being a weak zero energy process (which we show later) since $M$ and $W$ are local martingales.
	
	If $F(t,x)\in C^{0,2}$ and $\gamma(t,x),\beta(t,x) \in C^{0,1}$ with $\alpha$-H{\"o}lder-continuous derivatives in $x$, then Flandoli and Russo \cite[p.279]{flandoli2002generalized} give us the It{\^o}-Ventzell formula for weak Dirichlet processes:
	\begin{align*}
		F(t,X_t) =\ &  F(0,X_0) + \int_0^t \beta (r,X_r) \diff r + \int_0^t \gamma (r,X_r) \diff W_r + \int_0^t F_x (r,X_r) \diff^{-} X_r \\
		&+ \int_0^t \gamma_x (r,X_r) \diff [X,W]_r + \frac{1}{2} \int_0^t F_{xx} (r,X_r) \diff [X]_r.
	\end{align*}
	Next note that $\int_0^t F_x (r,X_r) \diff^{-} X_r$ exists. Due to $\int_0^t F_x (r,X_r) \diff^{-} M_r = \int_0^t F_x (r,X_r) \diff M_r$ coinciding with the classical stochastic integral, $\int_0^t F_x (r,X_r) \diff^{-} A_r$ also exists. Moreover, when carefully scanning the to our setting adapted proof of the It{\^o}-Ventzell formula in Flandoli and Russo \cite[pp.279-280]{flandoli2002generalized}, we can reduce the regularity assumptions of $\beta$ and $\gamma$ due to Lemma \ref{lem: substitution}. Thus, it is sufficient to assume that $\beta(t,x) \in C^{0,1}$ (without H{\"o}lder-continuity of $\beta_x(t,x)$ in $x$) and $\gamma(t,x) \in C^{0,1+\alpha}$ (with only local $\alpha$-H{\"o}lder-continuity of $\gamma_x(t,x)$ in $x$) in case (\ref{th, item: proof b}).
	
	Hence, it remains to prove the case (\ref{th, item: proof a}) where $F \in C^{0,1}$ and the Assumptions 1-3 hold:\\
	First of all, since $X$ is continuous, we can use a classical localization argument and assume that there exists a compact set $K \subset \Real$ such that $X_t (\omega) \in K$ for all $t \in [0,T]$ and almost all $\omega \in \Omega$. Then, by using again classical localization arguments
	, we can assume that $\sup_{x \in K} |\beta (t,x)| \in \LL^{1}(\Omega \times [0,T],\diff \PP \times \diff t)$, $\sup_{x \in K} |\gamma (t,x)| \in \LL^{2}(\Omega \times [0,T],\diff \PP \times \diff t)$ both from Definition \ref{def: stochastic flow}(\ref{def, item: integrability}). Now, the prerequisites of Lemma \ref{lem: substitution} are fulfilled, and thus 
	all of the (stochastic) integrals in this proof are well-defined.
	
	Let $N$ be a continuous local martingale. Since $N$, $M$, and $W$ are semimartingales, we know:
	\begin{align*}
		\left[\int_0^{\cdot} F_x (r,X_r) \diff M_r, N \right] &= \int_0^{\cdot} F_x (r,X_r) \diff \left[M, N \right]_r,\\
		\left[\int_0^{\cdot} \gamma (r,X_r) \diff W_r, N \right] &= \int_0^{\cdot} \gamma (r,X_r) \diff \left[W, N \right]_r.
	\end{align*}
	Hence, it remains to show: 
	\begin{align*}
		\left[ F(\cdot,X),N\right] = \int_0^{\cdot} F_x (r,X_r) \diff \left[M,N\right]_r + \int_0^{\cdot} \gamma (r,X_r) \diff \left[W,N\right]_r.
	\end{align*}
	Using Definition \ref{def: quadratic variation alternative definition}, we consider the $ucp$ limit of
	\begin{align*}
		\int_0^t &\left[F(r+\varepsilon,X_{r+\varepsilon})-F(r,X_{r})\right] \dfrac{N_{r+\varepsilon}-N_{r}}{\varepsilon} \diff r \\
		&= \int_0^t \left[F(r+\varepsilon,X_{r+\varepsilon})-F(r+\varepsilon,X_{r})\right] \dfrac{N_{r+\varepsilon}-N_{r}}{\varepsilon} \diff r + \int_0^t \left[F(r+\varepsilon,X_{r})-F(r,X_{r})\right] \dfrac{N_{r+\varepsilon}-N_{r}}{\varepsilon} \diff r \\
		&=: I_1 (t,\varepsilon) + I_2 (t,\varepsilon).
	\end{align*}
	We first show: $\lim_{\varepsilon \to 0} I_1 (t,\varepsilon) = \int_0^t F_x (r,X_r) \diff \left[M,N\right]_r$. It is
	\begin{align*}
		I_1 (t,\varepsilon) &= \int_0^t F_x(r,X_{r})\left(X_{r+\varepsilon}-X_{r}\right) \dfrac{N_{r+\varepsilon}-N_{r}}{\varepsilon} \diff r \\
		&\hspace{12pt}+ \int_0^t \left(F_x(r+\varepsilon,X_{r})-F_x(r,X_{r})\right) \left(X_{r+\varepsilon}-X_{r}\right) \dfrac{N_{r+\varepsilon}-N_{r}}{\varepsilon} \diff r \\
		&\hspace{12pt}+ \int_0^t \int_0^1 \left\{ F_x (r+\varepsilon,X_{r}+\lambda(X_{r+\varepsilon}-X_{r})) - F_x(r+\varepsilon,X_{r}) \right\} \diff \lambda \left(X_{r+\varepsilon}-X_{r}\right)\dfrac{N_{r+\varepsilon}-N_{r}}{\varepsilon} \diff r \\
		&=: I_{1,1} (t,\varepsilon) + R_{1,1} (t,\varepsilon) + R_{1,2} (t,\varepsilon).
	\end{align*}
	We know that $r \mapsto F_x(r,X_r)$ is c{\`a}dl{\`a}g since $F \in C^{0,1}$ and $X$ is a.s. continuous. Thus, Lemma \ref{rem: cadlag ucp} implies that $\lim_{\varepsilon \to 0} I_{1,1} (t,\varepsilon) = \int_0^t F_x (r,X_r) \diff \left[X,N\right]_r = \int_0^t F_x (r,X_r) \diff \left[M,N\right]_r$ since $[A,N]_r\equiv 0$. Hence, it remains to show: $\lim_{\varepsilon \to 0} R_{1,1} (t,\varepsilon) = 0$ and $\lim_{\varepsilon \to 0} R_{1,2} (t,\varepsilon) = 0$ in $ucp$. 
	For this purpose, we define for $r \in [0,T]$ and $\varepsilon>0$:
	\begin{align*}
		Y_r^{\varepsilon} &:= F_x(r+\varepsilon,X_{r})-F_x(r,X_{r}), \\
		Z_r^{\varepsilon} &:= \int_0^1 \left\{F_x (r+\varepsilon,X_{r}+\lambda(X_{r+\varepsilon}-X_{r})) - F_x(r+\varepsilon,X_{r}) \right\}\diff \lambda.
	\end{align*}
	Let $\omega \in \Omega$ be arbitrary such that $r \to X_r (\omega)$ is continuous. Hence, there exists a convex compact set $K$ such that $X_r (\omega) \in K$ for all $r \in [0,T]$. Since $F_x(t,x)$ is continuous, we know that $F_x(t,x)$ is uniformly continuous on $[0,T] \times K$. Thus, for all $L>0$ there exists a $0<\rho_{\omega}\leq1$ such that it holds for all $r \in [0,T]$ and for all $0<\varepsilon \leq \rho_{\omega}$: $\left| Y_r^{\varepsilon} (\omega) \right| \leq L$. (Note that $\norm{(r+\varepsilon,X_r(\omega))-(r,X_r(\omega))}_2 = \varepsilon \leq \rho_{\omega}$.) 
	In particular, we get for any arbitrary $L>0$ a $0<\rho_{\omega}\leq1$ such that $\sup_{r \in [0,T]} \left| Y_r^{\varepsilon} (\omega) \right| \leq L$ for all $0<\varepsilon \leq \rho_{\omega}$. Hence, $\sup_{r \in [0,T]} \left| Y_r^{\varepsilon} (\omega) \right| \xrightarrow[a.s.]{\varepsilon \to 0} 0$. \\
	Next, note that it holds that $X_{r}(\omega)+\lambda(X_{r+\varepsilon} (\omega) -X_{r}(\omega)) \in K$ for all $\lambda \in [0,1]$ since $X_r(\omega), \linebreak X_{r+\varepsilon}(\omega) \in K$. Moreover, $X_r (\omega)$ is uniformly continuous in $r$ for $r \in [0,T]$, i.e., for all $\tilde{L}>0$ there exists a $0<\tilde{\rho}_{\omega}\leq1$ such that for all $r \in [0,T]$ and for all $0<\varepsilon \leq \tilde{\rho}_{\omega}$ it holds: $\left| X_{r+\varepsilon}(\omega) -X_r(\omega) \right| \leq \tilde{L}$. Let $L>0$ arbitrary. 
	Since $F_x (t,x)$ is uniformly continuous on $[0,T] \times K$, there exists a $\delta>0$ such that for all $x,y \in K$ with $\norm{(t,x)-(t,y)}_2 \leq \delta$ it holds that $\left| F_x (t,x)-F_x(t,y)\right| \leq L$ for all $t \in [0,T]$ fixed. Now choose $\tilde{L} = \delta$ and define $\rho_{\omega} := \min \{\delta,\tilde{\rho}_{\omega} \}$. Then \[\norm{(r+\varepsilon,X_{r} (\omega)+\lambda (X_{r+\varepsilon} (\omega)-X_{r} (\omega))) - (r+\varepsilon,X_{r} (\omega)) }_2 = \lambda |X_{r+\varepsilon} (\omega)-X_{r} (\omega)| \leq \delta\] for all $0<\varepsilon \leq \rho_{\omega}$, all $\lambda \in [0,1]$, and all $r \in [0,T]$. 
	Hence, it holds that \[\left|F_x (r+\varepsilon,X_{r}(\omega)+\lambda(X_{r+\varepsilon}(\omega)-X_{r}(\omega))) - F_x(r+\varepsilon,X_{r}(\omega))\right| \leq L \] for all $0<\varepsilon \leq \rho_{\omega}$, all $\lambda \in [0,1]$, and all $r \in [0,T]$. This implies that $\left| Z_r^{\varepsilon} (\omega) \right| \leq L$ for all $0<\varepsilon \leq \rho_{\omega}$ and all $r \in [0,T]$. In particular, we get $\sup_{r \in [0,T]} \left| Z_r^{\varepsilon} (\omega) \right| \leq L$ for all $0<\varepsilon \leq \rho_{\omega}$. 
	Since we chose $L>0$ and $\omega \in \Omega$ arbitrary, $\sup_{r \in [0,T]} \left| Z_r^{\varepsilon} \right| \xrightarrow[a.s.]{\varepsilon \to 0} 0$. Now, Lemma \ref{lem: ucp convergence double} yields that $\lim_{\varepsilon \to 0} R_{1,1} (t,\varepsilon) = 0$ and $\lim_{\varepsilon \to 0} R_{1,2} (t,\varepsilon) = 0$ in $ucp$.
	
	Let us next show that $\lim_{\varepsilon \to 0} I_2 (t,\varepsilon) = \int_0^t \gamma (r,X_r) \diff \left[W, N \right]_r$. We use \eqref{eq: definition formula F} combined with Lemma \ref{lem: substitution} (whose conditions are fulfilled due to the localization arguments at the beginning) in $I_2$ to write:
	\begin{align*}
		I_2 (t,\varepsilon) &= \int_0^t \int_{r}^{r+\varepsilon} \beta (s,X_r) \diff s \dfrac{N_{r+\varepsilon}-N_{r}}{\varepsilon} \diff r + \int_0^t \int_{r}^{r+\varepsilon} \gamma (s,X_r) \diff W_s \dfrac{N_{r+\varepsilon}-N_{r}}{\varepsilon} \diff r \\
		&= \int_0^t \int_{r}^{r+\varepsilon} \gamma (r,X_{r}) \diff W_s \dfrac{N_{r+\varepsilon}-N_{r}}{\varepsilon} \diff r + \int_0^t \int_{r}^{r+\varepsilon} \beta (s,X_r) \diff s \dfrac{N_{r+\varepsilon}-N_{r}}{\varepsilon} \diff r \\
		&\hspace{12pt} + \int_0^t \int_{r}^{r+\varepsilon} \left[ \gamma (s,X_r) - \gamma(r,X_{r})\right] \diff W_s \frac{N_{r+\varepsilon}-N_{r}}{\varepsilon} \diff r \\
		&=: I_{2,1} (t,\varepsilon) + R_{2,1} (t,\varepsilon) + R_{2,2} (t,\varepsilon).
	\end{align*}
	For $I_{2,1}$, we get due to Lemma \ref{rem: cadlag ucp} and Assumption \ref{th, item: gamma continuity}:
	\begin{align*}
		I_{2,1} (t,\varepsilon) = \int_0^t \gamma (r,X_{r}) \dfrac{(W_{r+\varepsilon}-W_{r})(N_{r+\varepsilon}-N_{r})}{\varepsilon} \diff r \xrightarrow[ucp]{\varepsilon \to 0}  \int_0^t \gamma (r,X_r) \diff \left[W, N \right]_r.
	\end{align*}
	For the two terms remaining, we define for $r \in [0,T]$ and $\varepsilon>0$:
	\begin{align*}
		\widetilde{Y}^{\varepsilon}_r &:= \dfrac{1}{\varepsilon}\int_{r}^{r+\varepsilon} \beta (s,X_r) \diff s, &
		\widetilde{Z}^{\varepsilon}_r &:= \dfrac{1}{\sqrt{\varepsilon}} \int_{r}^{r+\varepsilon} \left[ \gamma (s,X_r) - \gamma(r,X_{r})\right] \diff W_s.
	\end{align*}
	First, we note that both $\widetilde{Y}^{\varepsilon}$ and $\widetilde{Z}^{\varepsilon}$ are a.s. bounded for all $\varepsilon>0$ since $\sup_{x \in K} |\beta (t,x)| \in \LL^{1}(\Omega \times [0,T],\diff \PP \times \diff t)$, $\sup_{x \in K} |\gamma (t,x)| \in \LL^{2}(\Omega \times [0,T],\diff \PP \times \diff t)$, and $X$ is a.s. continuous.\\
	We continue with $R_{2,1}$: In the case of Assumption \ref{th, item: beta integrability}, classical localization arguments give us the existence of a $\tilde{B}_1>0$ such that $\int_0^T \sup_{s \in [0,T]} |\beta (s,X_t)| \diff t \leq \tilde{B}_1$ a.s. and we get a.s.:
	\begin{align*}
		\int_0^T |\widetilde{Y}^{\varepsilon}_r| \diff r \leq \int_0^T \dfrac{1}{\varepsilon} \int_{r}^{r+\varepsilon} \left|\beta (s,X_r)\right| \diff s \diff r \leq \int_0^T \sup_{s \in [0,T]} |\beta (s,X_r)| \diff r \leq \tilde{B}_1.
	\end{align*}
	In the case of Assumption \ref{th, item: beta integral expected value}, we get:
	\begin{align*}
		\EX \left[ \left| \int_0^T |\widetilde{Y}^{\varepsilon}_r| \diff r \right| \right] &\leq \EX \left[ \int_0^T \dfrac{1}{\varepsilon} \int_{r}^{r+\varepsilon} \left|\beta (s,X_r)\right| \diff s \diff r \right] \\
		&= \int_0^T \dfrac{1}{\varepsilon} \int_{r}^{r+\varepsilon} \EX \left[ \left| \beta (s,X_r) \right|  \right] \diff s \diff r  \leq \int_0^T \sup_{s \in [0,T]} \EX \left[ \left| \beta(s,X_r)\right| \right] \diff r \leq B_1,
	\end{align*}
	i.e., $\int_0^T |\widetilde{Y}^{\varepsilon}_r| \diff r$ is bounded in $\LL^1(\diff \PP)$. Hence, in both cases it holds that $\int_0^T |\widetilde{Y}^{\varepsilon}_r| \diff r$ is bounded in probability and Lemma \ref{lem: ucp convergence single L1} implies that $\lim_{\varepsilon \to 0} R_{2,1} = 0$ in $ucp$. 
	
	Finally, we analyze $R_{2,2}$: First, we get for $r \in [0,T]$ fixed:
	\begin{align*}
		\widetilde{Z}^{\varepsilon}_r &= \dfrac{W_{r+\varepsilon}-W_{r}}{\sqrt{\varepsilon}} \cdot \dfrac{1}{W_{r+\varepsilon}-W_{r}} \int_{r}^{r+\varepsilon} \left[ \gamma (s,X_r) - \gamma(r,X_{r})\right] \diff W_s.
	\end{align*}
	The first fraction is $\mathcal{N} (0,1)$-distributed for all $\varepsilon$, and hence converges in distribution to $Z \sim \mathcal{N} (0,1)$ for $\varepsilon \to 0$. The second part converges in probability to $0$ for $\varepsilon \to 0$ due to Protter \cite[p.96]{Protter}. Note that we are allowed to use that result since, for $s \geq r$, $H_s := \gamma (s,X_r) - \gamma(r,X_{r})$ is adapted and continuous due to $\gamma$ being continuous in the first argument. Hence, Slutsky implies that $\widetilde{Z}^{\varepsilon}_r \xrightarrow[d]{\varepsilon \to 0} 0$. Due to $0$ being constant, this convergence also holds in probability.\\
	In the next step, we want to show that this convergence also holds in $L^2 (\diff \PP)$. For this purpose, we show that $|\widetilde{Z}^{\varepsilon}_r |^2$ is uniformly integrable. Let $\varrho>0$ be arbitrary. 
	Then, we get using Mao\cite[p.39]{Mao} (with $p=2+\varrho$):
	\begin{align*}
		\EX [|\widetilde{Z}^{\varepsilon}_r |^{2+\varrho}] &= \dfrac{1}{\varepsilon^{1+\frac{\varrho}{2}}} \EX [\left| \int_{r}^{r+\varepsilon} \left[ \gamma (s,X_r) - \gamma(r,X_{r})\right] \diff W_s \right|^{2+\varrho}] \\
		&\leq \dfrac{1}{\varepsilon^{1+\frac{\varrho}{2}}} \left(\dfrac{(2+\varrho)(1+\varrho)}{2}\right)^{\frac{2+\varrho}{2}} \varepsilon^{\frac{\varrho}{2}} \EX \left[ \int_{r}^{r+\varepsilon} \left|\gamma (s,X_r) - \gamma(r,X_{r}) \right|^{2+\varrho} \diff s \right].
	\end{align*}
	We define $\tilde{C} := \left(\frac{(2+\varrho)(1+\varrho)}{2}\right)^{\frac{2+\varrho}{2}}$ and $C := \tilde{C} \cdot 2^{2+\varrho}$. In particular, $\tilde{C}\leq C$. 
	Now, we get for all $r \in [0,T]$ using Lemma \ref{lem: show old assumption}:
	\begin{align} \label{eq: 2rho moment 2b}
		\EX [|\widetilde{Z}^{\varepsilon}_r |^{2+\varrho}] &\leq \tilde{C} \cdot \dfrac{1}{\varepsilon} 	\int_r^{r+\varepsilon} \EX \left|\gamma (s,X_r) - \gamma(r,X_{r}) \right|^{2+\varrho} \diff s \notag \\
		&\leq \tilde{C} \sup_{s \in [0,T]} \EX [\left|\gamma (s,X_r) - \gamma(r,X_{r}) \right|^{2+\varrho}] \leq C \sup_{s \in [0,T]} \EX [\left|\gamma (s,X_r)\right|^{2+\varrho}] \leq C \cdot B_2.
	\end{align}
	Thus, $\widetilde{Z}^{\varepsilon}_r \xrightarrow[\LL^2(\diff \PP)]{\varepsilon \to 0} 0$ for all $r \in [0,T]$. 
	Analogously to \eqref{eq: 2rho moment 2b}, we also get $\sup_{r \in [0,T]} \EX [|\widetilde{Z}^{\varepsilon}_r |^{2}] \leq C \cdot B_2$. 
	Hence, the bounded convergence theorem gives us that $\int_0^T \EX |\widetilde{Z}^{\varepsilon}_r |^{2} \diff r \xrightarrow{\varepsilon \to 0} 0$.
	Thus, 
	$\int_0^T |\widetilde{Z}^{\varepsilon}_r |^{2} \diff r \xrightarrow[\PP]{\varepsilon \to 0} 0$ and Lemma \ref{lem: ucp convergence single} implies that $\lim_{\varepsilon \to 0} R_{2,2} = 0$ in $ucp$ which gives us the proof of the Theorem. 
\end{myproof}

\section{Applications} \label{sec: applications}

In this section, we apply Theorem \ref{th: main result} to derive a novel stochastic calculus formula for quadratic variations and a representation result for strong solutions of SPDEs.

\subsection{Formulas for quadratic variations}

In this subsection, we give in Proposition \ref{prop: new quadratic variation formula} a novel stochastic calculus formula for quadratic variations. Afterwards, we point out some special cases. Note that we cannot use the It{\^o}-Ventzell formula in $C^{0,2}$ or $C^{1,2}$ since the function $F$ defined below is not twice continuously differentiable in the space direction. 
We remark further that Proposition \ref{prop: new quadratic variation formula} is not a special case of the Meyer-It{\^o} formula (see, for instance, Protter \cite{Protter}) which is only applicable for delta-convex functions, i.e., for the difference of two continuous convex functions. However, the class of $C^1$-functions and the class of delta-convex functions are not included in each other, see, e.g., Zelen{\'y} \cite{zeleny2002example}.

\begin{proposition} \label{prop: new quadratic variation formula}
	Let $W$ be a one-dimensional Brownian Motion, $X = M+A$ be a continuous weak Dirichlet FQV process, and $f: \Omega \times [0,T] \times \Real \to \Real$ be such that $f(\cdot,x)$ is previsible for all $x \in \Real^d$ fixed, $f(t,W_t)$ is a.s. c{\`a}dl{\`a}g, and $f$ is continuously differentiable in $x$ with the derivative being H{\"o}lder-continuous (for an $\alpha>0$) for all $t \in [0,T]$. Moreover, we assume that for all $K \subset \Real$ compact exists a $C_K(\omega)\geq 0$ with $\EX[C_K^2]<\infty$ such that $\sup_{(t,x) \in [0,T] \times K} (|f(\omega,t,x)| + |f_x(\omega,t,x)|) \leq C_K(\omega)$. Then, it holds that
	\begin{align*}
		\left[\int_0^t f(r,x) \diff W_r \big|_{x=X_t},W_t \right] = \int_{0}^{t} f(r,X_r) \diff r + \int_{0}^{t} \int_0^r f_x(s,x) \diff W_s \big|_{x=X_r} \diff [M,W]_r.
	\end{align*}
\end{proposition}

\begin{proof}
	First, we define $F(t,x) := \int_0^t f(r,x) \diff W_r$ (where we suppress the $\omega$ as usual). Then, $F$ is a stochastic flow in the sense of Definition \ref{def: stochastic flow} with local characteristics $\beta(t,x)=0$ and $\gamma(t,x)=f(t,x)$. The assumptions of Theorem \ref{th: main result} are fulfilled since continuous differentiability implies local Lipschitz continuity. Moreover, it holds that $F_x(t,x) = \int_0^t f_x(r,x) \diff W_r$ due to the stochastic Leibniz integral rule. Then, Theorem \ref{th: main result} implies with the decomposition $X=M+A$:
	\begin{align*}
		\int_0^t f(r,x) \diff W_r \big|_{x=X_t} = 0 + \int_0^t f(r,X_r) \diff W_r + \int_0^t \int_0^r f_x(s,x) \diff W_s \big|_{x=X_r} \diff M_r + \mathcal{B}^X(F)_t,  
	\end{align*}
	where $\mathcal{B}^X(F)_t$ is a weak zero energy process. Hence, we get for any continuous local martingale $\tilde{M}$:
	\begin{align*}
		\left[\int_0^t f(r,x) \diff W_r \big|_{x=X_t},\tilde{M}_t \right] = \int_0^t f(r,X_r) \diff [W,\tilde{M}]_r + \int_0^t \int_0^r f_x(s,x) \diff W_s \big|_{x=X_r} \diff [M,\tilde{M}]_r.
	\end{align*}
	Finally, choosing $\tilde{M}:=W$ yields the proposition since $\diff [W]_r = \diff r$.
\end{proof}

\begin{corollary} \label{cor: quadratic variation formula}
	Let $W$ be a one-dimensional Brownian Motion, and $f: \Real \to \Real$ be a continuously differentiable function. Then, we get that:
	\begin{align*}
		[f(W_t)W_t,W_t] = \int_{0}^{t} f(W_r) \diff r + \int_{0}^{t} f'(W_r) W_r \diff r.
	\end{align*}
\end{corollary}

\begin{proof}
	This corollary follows immediately from Proposition \ref{prop: new quadratic variation formula} after setting $X=W$ with $\diff [W]_r = \diff r$. 
	Moreover, note that $\int_{0}^{t} \int_0^r f_x(x) \diff W_s \big|_{x=X_r} \diff [M,W]_r = \int_{0}^{t} f'(X_r) W_r \diff [M,W]_r$ in this setting.
\end{proof}

For Corollary \ref{cor: quadratic variation formula}, we could have alternatively tried to use an It{\^o}-formula with weaker differentiability conditions.

\subsection{A representation result for strong solutions of SPDEs}

In this subsection, the overall target is to give a representation result for strong solutions of stochastic elliptic PDE for a special type of weak Dirichlet processes. We start by giving a generalization of the solution theory of SPDEs (naturally generalizing the theory of Gozzi and Russo \cite{Gozzi} who considered PDEs). 

Let $b: [0,T] \times \Real^d \to \Real^d$ and $\sigma: [0,T] \times\Real^d \to \Real^{d \times d}$ be two continuous functions. Then, we define the time-dependent elliptic operator $L_0 (t): D(L_0) \subset C^0 ([0,T] \times\Real^d) \to C^0 ([0,T] \times\Real^d)$ with $D(L_0) = C^{0,2} ([0,T] \times\Real^d)$ by $L_0(t) F(t,x) := b(t,x) F_x(t,x) + \frac{1}{2} tr(\sigma^\intercal(t,x) F_{xx}(t,x) \sigma(t,x)$, where $tr$ denotes the trace of a matrix. We then consider the stochastic inhomogeneous time-dependent elliptic problem for all $(t,x) \in [0,T] \times \Real^d$ and almost all $\omega \in \Omega$:
\begin{align} \label{def: elliptic pde}
	L_0(t) F(\omega,t,x) = h(\omega,t,x), 
\end{align}
where $h(\omega,\cdot) \in C^0([0,T] \times \Real^d)$ almost surely.

\begin{definition}
	We say that $F$ with local characteristics $(\beta(t,x),\gamma(t,x))$ is a strict solution to the stochastic elliptic problem \eqref{def: elliptic pde} if for a.e. $\omega \in \Omega$ it holds that $F(\omega,\cdot) \in D(L_0)$, $\beta(\omega,\cdot) \in C^{0,1}([0,T] \times \Real)$, $\gamma(\omega,\cdot) \in C^{0,1+\alpha}([0,T] \times\Real^d)$ for an $\alpha \in (0,1)$, and \eqref{def: elliptic pde} holds.
\end{definition}

Analogously to Gozzi and Russo \cite{Gozzi} who defined a strong solution for PDEs, we define a strong solution for SPDEs as follows:

\begin{definition} \label{def: strong solution}
	We say that a stochastic flow $F$ with local characteristics $(\beta( {t},x),\gamma( {t},x))$ is a strong solution to the stochastic elliptic problem \eqref{def: elliptic pde} if there exists a sequence $(F^n)_{n \in \N}$ with local characteristics $(\beta^n(\omega,t,x),\gamma^n(\omega,t,x))$ and a sequence $(h^n)_{n \in \N}$ such that for every $n \in \N$ and a.e. $\omega \in \Omega$ holds:
	\begin{enumerate}
		\myitem{1} $F^n (\omega,\cdot) \in D(L_0)$, and $h^n (\omega,\cdot) \in C^0([0,T] \times \Real^d)$,
		\myitem{2} $\beta^n(\omega,\cdot) \in C^{0,1}([0,T] \times \Real)$, and $\gamma^n(\omega,\cdot) \in C^{0,1+\alpha}([0,T] \times \Real^d)$ for an $\alpha \in (0,1)$,
		\myitem{3} $F^n$ is a strict solution of $L_0(t) F^n(\omega,t,x) = h^n(\omega,t,x)$ for all $(t,x) \in [0,T] \times \Real^d$ and a.e. $\omega \in \Omega$,
		\myitem{4} \label{defitem: strong solution Fh} $F^n (\omega,\cdot) \xrightarrow{n \to \infty} F(\omega,\cdot)$ and $h^n (\omega,\cdot) \xrightarrow{n \to \infty} h(\omega,\cdot)$ in $C^0([0,T] \times \Real^d)$ (equipped with the supremum-norm on compact sets), and
		\myitem{5} $\beta^n (\omega,\cdot) \xrightarrow{n \to \infty} \beta (\omega,\cdot)$ in $\LL^1(K;\diff t)$ for all compact sets $K \subset [0,T] \times \Real$.
	\end{enumerate}
\end{definition}


Note that this a generalization of the notion of a strict solution since, of course, every strict solution is also a strong solution by taking the constant sequences $F^n := F$ and $h^n := h$. Our representation theorem states then:

\begin{theorem}
	Let $b: [0,T] \times \Real^d \to \Real^d$ and $\sigma: [0,T] \times \Real^d \to \Real^{d \times d}$ be two continuous functions, and $F$ (with local characteristics $(\beta(\omega,t,x),\gamma(\omega,t,x))$) be a strong solution of problem \eqref{def: elliptic pde} where we denote the approximation sequence by $F^n$ (with local characteristics $(\beta^n(\omega,t,x),\gamma^n(\omega,t,x))$). Define $X_t := x_0 + \int_{0}^t \sigma (r,X_r) \diff W_r + A_t$, where $x_0 \in \Real^d$ and $A$ is a continuous weak zero energy process such that $X$ is a continuous FQV process. Furthermore, we assume:
	\begin{enumerate}
		\myitem{1} $F(\omega,\cdot) \in C^{0,1}([0,T] \times \Real^d)$ for a.e. $\omega \in \Omega$,
		\myitem{2} $\beta$ and $\gamma$ fulfill the assumptions of Theorem \ref{th: main result} with $\gamma(\omega,\cdot) \in C^{0,1} ([0,T] \times \Real^d)$,
		\myitem{3} $\gamma^n$ satisfy Assumption \ref{th, item: gamma lemma} of Theorem \ref{th: main result} for all $n \in \N$, and
		\myitem{4} it holds in $ucp$ that
		\begin{align} \label{eq: elliptic pde assumption}
			\lim_{n \to \infty} \Big|\int_0^t (F_x (r,X_r)-F^n_x (r,X_r)) \diff^- A_r \Big| &+ \Big|\int_0^t \tr((\gamma_x (r,X_r)-\gamma^n_x (r,X_r)) \sigma(r,X_r)) \diff r \Big| \notag \\
			&+ \Big|\int_0^t (F_x (r,X_r)-F^n_x (r,X_r)) b(r,X_r) \diff r \Big| = 0.
		\end{align}
		In particular, all of these integrals are assumed to exist.
	\end{enumerate}
	Then, we have that (suppressing the dependence on $\omega$)
	\begin{align*}
		F(t,X_t) =& F(0,X_0) + \int_0^t \gamma(r,X_r) \diff W_r + \int_0^t F_x(r,X_r) \sigma(r,X_r) \diff W_r + \int_0^t \beta (r,X_r) \diff r \\
		&+ \int_0^t F_x (r,X_r) \diff^- A_r + \int_0^t \tr(\gamma_x (r,X_r) \sigma(r,X_r)) \diff r
		+ \int_0^t h (r,X_r) \diff r \\
		&- \int_0^t b(r,X_r) F_x (r,X_r) \diff r.
	\end{align*}
\end{theorem}

\begin{proof}
	Since $F$ is a strong solution of \eqref{def: elliptic pde} with approximating sequences $F^n$ and $h^n$, we get with the second part of Theorem \ref{th: main result} (suppressing the dependence on $\omega$):
	\begin{align}
		F^n(t,X_t) =& F^n (0,X_0) + \int_0^t \gamma^n (r,X_r) \diff W_r + \int_0^t F^n_x(r,X_r) \sigma(r,X_r) \diff W_r + \int_0^t \beta^n (r,X_r) \diff r \notag \\
		&+ \int_0^t F^n_x (r,X_r) \diff^- A_r + \int_0^t \tr(\gamma^n_x (r,X_r) \sigma(r,X_r)) \diff r \notag \\
		&+ \frac{1}{2} \int_0^t \tr(\sigma^\intercal (r,X_r) F^n_{xx} (r,X_r) \sigma(r,X_r)) \diff r \notag \\
		=& F^n (0,X_0) + \int_0^t \gamma^n (r,X_r) \diff W_r + \int_0^t F^n_x(r,X_r) \sigma(r,X_r) \diff W_r + \int_0^t \beta^n (r,X_r) \diff r \notag \\ 
		&+ \int_0^t F^n_x (r,X_r) \diff^- A_r + \int_0^t \tr(\gamma^n_x (r,X_r) \sigma(r,X_r)) \diff r + \int_0^t L_0(r) F^n (r,X_r) \diff r \notag \\
		&- \int_0^t b(r,X_r) F^n_x (r,X_r) \diff r. \label{eq: elliptic pde c2}
	\end{align}
	Now, we define:
	\begin{align}
		M_t^n :=& F^n(t,X_t)-F^n (0,X_0) - \int_0^t \beta^n (r,X_r) \diff r - \int_0^t F^n_x (r,X_r) \diff^- A_r - \int_0^t \tr(\gamma^n_x (r,X_r) \sigma(r,X_r)) \diff r \notag \\
		&- \int_0^t L_0(r) F^n (r,X_r) \diff r + \int_0^t b(r,X_r) F^n_x (r,X_r) \diff r, \notag \\
		M_t :=& F(t,X_t)-F (0,X_0) - \int_0^t \beta (r,X_r) \diff r - \int_0^t F_x (r,X_r) \diff^- A_r - \int_0^t \tr(\gamma_x (r,X_r) \sigma(r,X_r)) \diff r \notag \\
		&- \int_0^t h (r,X_r) \diff r + \int_0^t b(r,X_r) F_x (r,X_r) \diff r. \label{eq: mt equation}
	\end{align}
	Note that for almost all $\omega \in \Omega$, it holds that $L_0(r) F^n (r,X_r(\omega)) = h^n(r,X_r(\omega)) \xrightarrow{n \to \infty} h (r,X_r(\omega))$ uniformly in $r$ due to \ref{defitem: strong solution Fh} in Definition \ref{def: strong solution} and $(r,X_r(\omega))$ taking values in a compact set for $r \in [0,t]$. Then, by the assumptions of a strong solution and by \eqref{eq: elliptic pde assumption}, it holds that $M_t^n \xrightarrow{n \to \infty} M_t$ in $ucp$. Moreover, it follows from \eqref{eq: elliptic pde c2} that $M_t^n$ is a local martingale. Proposition 4.4 of Gozzi and Russo \cite{Gozzi} implies then that $M$ is also a local martingale. Furthermore, note that $M_t$ is continuous in $t$ since $X$, $F$, and $A$ are continuous. Then, we define the weak zero energy process
	\begin{align*}
		\tilde{\mathcal{B}}^X (F)_t :=& \int_0^t \beta (r,X_r) \diff r + \int_0^t F_x (r,X_r) \diff^- A_r + \int_0^t \tr(\gamma_x (r,X_r) \sigma(r,X_r)) \diff r + \int_0^t h (r,X_r) \diff r \\
		&- \int_0^t b(r,X_r) F_x (r,X_r) \diff r.
	\end{align*}
	Rearranging \eqref{eq: mt equation} implies that $F(t,X_t) = F(0,X_0) + M_t + \tilde{\mathcal{B}}^X (F)_t$.	Next, we apply the first part of Theorem \ref{th: main result} on $F$ to get that $F(t,X_t) = F(0,X_0) + \int_0^t \gamma(r,X_r) \diff W_r + \int_0^t F_x(r,X_r) \sigma(r,X_r) \diff W_r + \mathcal{B}^X (F)_t$, where $\tilde{M}_t := \int_0^t \gamma(r,X_r) \diff W_r + \int_0^t F_x(r,X_r) \sigma(r,X_r) \diff W_r$ is a continuous local martingale. Finally, since $M_t$ and $\tilde{M}_t$ are continuous, the decomposition of the weak Dirichlet process $F(t,X_t)$ is unique. In particular, we get that $M_t=\tilde{M}_t$ and $\mathcal{B}^X (F)_t = \tilde{\mathcal{B}}^X (F)_t$ which implies the claim.
\end{proof}

\subsection{An example from financial mathematics}

In this subsection, we consider a simplified version of a model introduced in Bank and Baum \cite{bank2004hedging} which describes a wealth process of a large investor with reduced differentiability assumptions. The large investor directly influences the price with the amount of shares she/he buys/sells.

For this purpose, let $T>0$ and define for $0\leq t \leq T$ a price fluctuation process $P_t^\vartheta=P(t,\vartheta)$, $\vartheta \in \Real$, of the risky asset as a stochastic flow with local characteristics $(\beta(t,x),\gamma(t,x))$ as in Definition \ref{def: stochastic flow} with $\beta(t,x) \in C^{0,0}$ and $\gamma(t,x) \in C^{0,\alpha}$ for an $\alpha \in (0,1)$. The parameter $\vartheta$ denotes the number of shares the large investor holds. Moreover, we assume that $P$ fulfills the assumptions of Theorem \ref{th: main result}. (This is a relaxation compared to Bank and Baum \cite{bank2004hedging}, where second order differentiability is needed.) Let $X = (X_t)_{t \in [0,T]}$ be a (self-financing\footnote{That is, assuming that there are no intermediate capital injections.}) strategy for the large investor describing the number of shares held. Assuming that $X$ is a semimartingale, according to Bank and Baum \cite[p.5]{bank2004hedging}, the discounted bank account value, say $\Xi^X$, of the large investor evolves as 
\begin{align}
	\Xi_t^X = \Xi_0 - \int_0^t P(r,X_r) \diff X_r - [P(\cdot,X),X]_t, \quad \text{for $0\leq t \leq T$,} \label{eq: def bank}
\end{align}
where the integral models the fluctuations of the portfolio value until time $t$ and the quadratic variation term accounts for the sensitivity of the prices regarding the large investor's orders at time $t$. 
By splitting an order into smaller packages, Bank and Baum \cite{bank2004hedging} argue that when liquidating the portfolio, the amount obtained by the large investor is given by $L(t,\vartheta) = \int_0^\vartheta P(t,x) \diff x$, where $\vartheta$ denotes the large investor's number of shares before liquidation. The real (or asymptotically realizable) wealth $V^X$ which can be achieved by the strategy $X$ until time $t$ is consequently given by
\begin{align}
	V_t^X = \Xi_t^X + L(t,X_t). \label{eq: def wealth}
\end{align}

\begin{proposition}
	Let $X$ be a continuous self-financing strategy which is a semimartingale. The dynamics of the real wealth process $V_t^X$ is then given by:
	\begin{align*}
		V_t^X = V_0^X + \int_0^t \beta^L (r,X_r ) \diff r + \int_0^t \gamma^L (r,X_r) \diff X_r  - \frac{1}{2} \int_0^t P_x (r,X_r) \diff[X]_r,
	\end{align*}
	where $\beta^L (t,\vartheta) = \int_0^\vartheta \beta(t,x) \diff x$ and $\gamma^L (t,\vartheta) = \int_0^\vartheta \gamma(t,x) \diff x$.
\end{proposition}

\begin{proof}
	First, we can use the standard decomposition $X=M+A$ with $M$ being a martingale and $A$ a finite variation process since $X$ is a semimartingale. Next, we notice by the definition of $L$, that $L$ is also a stochastic flow with local characteristics $(\beta^L (t,x), \gamma^L (t,x))$ in the sense of Definition \ref{def: stochastic flow} and fulfills the assumption of Theorem \ref{th: main result} (b). Indeed, we are allowed to interchange the integrals when determining the local characteristics of $L$ since $\beta(s,x) \in C^{0,0}$ and thus $\beta$ is bounded on $[0,t]\times[0,\vartheta]$ for a.e. $\omega \in \Omega$ fixed. Furthermore, $\gamma(s,x) \in C^{0,\alpha}$, and we can therefore apply the stochastic Fubini theorem, given in Lemma 2.6 resp. Lemma 2.7 in Krylov \cite{krylov2009wentzell}), for the stochastic integral. Hence, we get with Theorem \ref{th: main result} (b):
	\begin{align}
		L(t,X_t) =& L(0,X_0) + \int_0^t \beta^L (r,X_r) \diff r + \int_0^t \gamma^L (r,X_r) \diff W_r + \int_0^t L_x (r,X_r) \diff M_r  \notag \\
		&+ \int_0^t L_x (r,X_r) \diff^{-} A_r + \int_0^t \gamma^L_x (r,X_r) \diff [X,W]_r + \frac{1}{2} \int_0^t L_{xx} (r,X_r) \diff [X]_r \notag \\
		=& L(0,X_0) + \int_0^t \beta^L (r,X_r) \diff r + \int_0^t \gamma^L (r,X_r) \diff W_r + \int_0^t P (r,X_r) \diff X_r \notag \\
		&+ \int_0^t \gamma (r,X_r) \diff [X,W]_r + \frac{1}{2} \int_0^t P_x (r,X_r) \diff [X]_r, \label{eq: ltxt}
	\end{align}
	since $L_x = P$ with $\gamma_x^L = \gamma$. Next, we apply Theorem \ref{th: main result} (a) to $P$ which gives us:
	\begin{align*}
		P(t,X_t) =& P(0,X_0) + \int_0^t \gamma (r,X_r) \diff W_r + \int_0^t P_x (r,X_r) \diff M_r  + \mathcal{B}^X (P)_t.
	\end{align*}
	Hence, it follows that:
	\begin{align*}
		[P(\cdot,X),X]_t &= \int_0^t \gamma (r,X_r) \diff [X,W]_r + \int_0^t P_x (r,X_r) \diff [X,M]_r \\
		&= \int_0^t \gamma (r,X_r) \diff [X,W]_r + \int_0^t P_x (r,X_r) \diff [X]_r,
	\end{align*}
	since $[A,M]_r = 0 = [A]_r$ due to $A$ being a finite variation process.
	Thus, we get for the bank account value $\Xi$ (defined in \eqref{eq: def bank}):
	\begin{align}
		\Xi_t^X = \Xi_0 - \int_0^t P(r,X_r) \diff X_r - \int_0^t \gamma (r,X_r) \diff [X,W]_r - \int_0^t P_x (r,X_r) \diff [X]_r. \label{eq: xit}
	\end{align}
	Next, we plug \eqref{eq: ltxt} and \eqref{eq: xit} into \eqref{eq: def wealth}, then several terms cancel out and the claim follows.
\end{proof}



\footnotesize
\bibliography{bibliography}
\footnotesize
\bibliographystyle{plain}

\footnotesize
\end{document}